\documentclass{amsart}
\usepackage{amsmath}
\usepackage{amssymb}
\usepackage{amsfonts}
\usepackage{graphicx}
\usepackage{tikz}
\usetikzlibrary{
	knots,
	hobby
}

\usepackage{longtable}
\usepackage{array}

\usepackage{lipsum} 
\usepackage[T1]{fontenc}
\newtheorem{theorem}{Theorem}
\newtheorem{lemma}{Lemma}
\newtheorem{example}{Example}
\newtheorem{remark}{Remark}
\newtheorem{definition}{Definition}

\tikzset{knot_diagram/.style={draw=black, thick, line cap=round}}
\tikzset{
	knot diagram/every strand/.append style={
		thick,
		black
  }
}
\tikzset{knot_arrow/.style={->}}
\tikzset{knot_arrow_rev/.style={<-}}

\title{Hierarchical quandles}
\author{Korablev Ph. G.}

\address{Chelyabinsk State University, Chelyabinsk, Russia; N.N. Krasovsky Institute of Mathematics and Meckhanics, Ekaterinburg, Russia}
\email{korablev@csu.ru}
\date{}

\begin{document}

\begin{abstract}
    We introduce the notion of a hierarchical quandle, which is a generalisation of diquandles and multi-quandles. Using hierarchical quandle colourings, we construct a cocycle invariants for links coloured by quandles.
\end{abstract}

\maketitle

\section{Introduction}

The notion of \emph{dichromatic links} was introduced in \cite{HK, HP}. The link in 3-sphere is called dichromatic if its components are coloured by two different colours. There is a natural extension of this notion --- \emph{$k$-coloured link} (see for example \cite{LS}). This is a link with components coloured by $k$ different colours. In the present paper we use a more general approach. We consider the colouring of the link by any finite quandle. This colouring is a homomorphism from the fundamental quandle of the link to the finite quandle. On the diagrammatic level, each colouring is just a classical colouring of the diagram arcs by elements of the finite quandle that satisfy the consistency condition in the neighbourhood of each crossing. From this point of view, each dichromatic or $k$-coloured link is just a link coloured by a trivial quandle. A similar approach to link colouring was used in \cite{NOR, Y} to construct an enhancement of the Kauffman bracket polynomial.

We introduce the notion of \emph{hierarchical quandle}, which is an algebraic structure with a special set of axioms. Like axioms for quandles motivated by Reidemeister moves of diagrams, axioms of hierarchical quandles are motivated by Reidemeister moves of coloured diagrams. In fact, we consider colourings by hierarchical quandle of diagram arcs coloured by quandle. That's why this structure is called hierarchical, because it's something like a quandle  over another quandle. \emph{Diquandles} (\cite{LSD}) and \emph{multi-quandles} (\cite{LS}) are the simplest examples of hierarchical quandles over trivial quandles.

In this paper we construct cocycle invariants for links coloured by finite quandles. The main concepts of these invariants are the same as for classical cocycle invariants (\cite{CJKLS}). In particular, we describe the method that allows to construct a chain complex for any finite hierarchical quandle.

The structure of the paper is as follows. The main goal of section \ref{Section:HierarchicalQuandles} is to introduce the notion of hierarchical quandles. In the subsection \ref{Subsection:Quandles} we recall the classical notion of a quandle. In subsection \ref{Subsection:FundamentalQuandle} we recall the well-known notion of a fundamental quandle, and also describe the representation of quandles by generators and relations. In subsection \ref{Subsection:QuandleColourings} we recall classical quandle colouring invariants for oriented links. In subsection \ref{Subsection:HierarchicalQuandle} we introduce the notion of hierarchical quandle. In subsection \ref{Subsection:HierarchicalQuandleColourings} we prove that the number of colourings by hierarchical quandle of the diagram of a link coloured by a quandle does not depend on the diagram of this link. So this number is an invariant of coloured links. The multi-set of these numbers (calculated for all colourings of the link) is an invariant of the link.

The section \ref{Section:CocycleInvariants} is devoted to the construction of cocycle invariants for links, coloured by quandle. In subsection \ref{Subsection:ChainGroups} we define chain groups, in subsection \ref{Subsection:BoundaryMaps} --- boundary homomorphisms, and in subsection \ref{Subsection:ChainComplex} --- the chain and the corresponding cochain complexes. In subsection \ref{Subsection:CocycleInvariant} we define the cocycle invariants. These invariants depends on a 2-cocycle of the constructed cochain complex. We prove that these invariants depends only on the cohomology class of the 2-cocycle.

In section \ref{Section:QuandlesForHierarchicalQuandles} we describe the method to construct a quandle from any hierarchical quandle over another quandle. This construction shows that our cocycle invariants are mostly the same as classical cocycle invariants. But there is a difference. If the quandle can be obtained from other quandle and hierarchical quandle, then our approach allows to split the colourings count invariant and cocycle invariant into several summands. In section \ref{Section:FurtherDevelopment} we discuss some tasks and directions for further development of the theory of hierarchical quandles.

\section{Hierarchical quandles}\label{Section:HierarchicalQuandles}

In this section we introduce the notion of hierarchical quandles and prove that the number of colourings of the link diagram by its elements is an invariant.

\subsection{Quandles}\label{Subsection:Quandles} The notion of quandle was introduced by Joice (\cite{J}) and Matveev (\cite{M}, he called it \emph{distributive groupoid}).

\begin{definition}
    \label{Definition:Quandle}
    The set $X$ with the binary operation $*\colon X\times X\to X$ is called \emph{quandle} if it satisfies to the following conditions:
    \begin{enumerate}
        \item For any $x\in X$: $x * x = x$;
        \item For any $a, b\in X$ there exists a unique $x\in X$ such that $x * a = b$;
        \item For any $x_1, x_2, x_3\in X$: $(x_1 * x_2) * x_3 = (x_1 * x_3) * (x_2 * x_3)$.
    \end{enumerate}

    If $(X, *)$ only satisfies the second and third conditions, then it is called \emph{rack}.
\end{definition}

The second axiom states that the equation $x * a = b$ for any $a, b\in X$ has a unique solution $x\in X$. Denote this solution by $x = b/a$. It's easy to see that $(X, /)$ is also a quandle.

\begin{definition}
    Let $(X, *)$ and $(A, \circ)$ be two quandles. The map $f\colon X\to A$ is called \emph{homomorphism} if for any $x_1, x_2\in X$: $f(x_1 * x_2) = f(x_1)\circ f(x_2)$. The homomorphism is called \emph{isomorphism} if it's a bijection.
\end{definition}

\begin{example}
    The simplest examples of quandles are the following.
    \begin{enumerate}
        \item \emph{Trivial quandle}. Let $X$ be a set. Define the quandle operation $*\colon X\times X\to X$ by the formula: $x_1 * x_2 = x_1$ for any $x_1, x_2\in X$.
        \item \emph{Conjugation quandle}. Let $G$ be a group. Define the quandle operation $*\colon G\times G\to G$ by the formula $x_1 * x_2 = x_2^{-1} x_1 x_2$ for any $x_1, x_2\in G$.
    \end{enumerate}
\end{example}

\subsection{Fundamental quandle}\label{Subsection:FundamentalQuandle} We will use the approach from \cite{M}. Each quandle can be represented by generators and relations. Let $A$ be a set, it will play the role of the alphabet. The word in the alphabet $A$ is a finite sequence of the elements of $A$ and the symbols $(, ), *, /$. Define the set of \emph{admissible} words inductively as follows:
\begin{enumerate}
    \item Any one-symbol word $a\in A$ is admissible;
    \item If the words $w_1$ and $w_2$ are admissible then the words $(w_1) * (w_2)$ and $(w_1) / (w_2)$ are also admissible.
\end{enumerate}

Denote the set of all admissible words in the alphabet $A$ by $W$. 

Let $R$ be a set of relations over the set $W$, i.e. the set of formal equations $r = s$, where $r, s\in W$. Define equivalence on the set $W$ as follows: two admissible words $w_1$ and $w_2$ are equivalent iff the word $w_1$ can be transformed into $w_2$ by the finite sequence of the following operations:
\begin{enumerate}
    \item $x * x \longleftrightarrow x$;
    \item $(x * y) / y \longleftrightarrow x$;
    \item $(x / y) * y \longleftrightarrow x$;
    \item $(x * y) * z \longleftrightarrow (x * z) * (y * z)$;
    \item $r \longleftrightarrow s$.
\end{enumerate}
Here $x, y, z\in W$ and $r, s$ are words in relations of the form $r = s$ from $R$. It's clear that the set of equivalence classes forms a quandle. Denote it $\Gamma_{A, R}$.

Let $K$ be an oriented knot or link, and let $D$ be a diagram of the link $K$. Denote $\mathcal{A}(D)$ the set of arcs of the diagram $D$ (i.e. the parts of the diagram between undercrossings). Define the quandle $\Gamma_{D}$ using generators and relations as follows. The set of generators is $\mathcal{A}(D)$. For each crossing of the diagram $D$ define the relation $a_1 * a_2 = a_3$, where $a_1, a_2, a_3$ are arcs as shown in the figure \ref{Figure:Crossings} (on the left for the positive crossing and on the right for the negative one). Denote the set of all these relations by $\mathcal{R}(D)$.

\begin{figure}[h]
    \begin{center}
        \begin{tikzpicture}[scale=0.75]
            \draw[knot_arrow, knot_diagram] (-1, -1) node[below left] {$a_2$} -- (1, 1);
            \draw[knot_diagram] (1, -1) node[below right] {$a_1$} -- (0.1, -0.1);
            \draw[knot_arrow, knot_diagram] (-0.1, 0.1) -- (-1, 1) node[above left] {$a_3$};
        \end{tikzpicture}
        \hspace{2cm}
        \begin{tikzpicture}[scale=0.75]
            \draw[knot_arrow, knot_diagram] (1, -1) node[below right] {$a_2$} -- (-1, 1);
            \draw[knot_diagram] (-1, -1) node[below left] {$a_3$} -- (-0.1, -0.1);
            \draw[knot_arrow, knot_diagram] (0.1, 0.1) -- (1, 1) node[above right] {$a_1$};
        \end{tikzpicture}
    \end{center}
    \caption{\label{Figure:Crossings}Arcs in the neighbourhood of positive (on the left) and negative (on the right) crossings}
\end{figure}
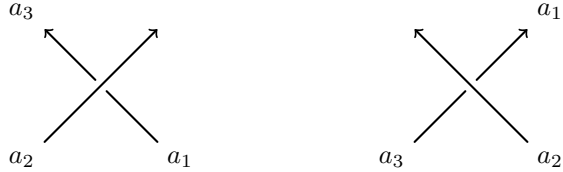

\begin{definition}
    Let $K$ be an oriented knot or link, and let $D$ be a diagram of the link $K$. The quandle $\Gamma_{\mathcal{A}(D), \mathcal{R}(D)}$ is called the \emph{fundamental quandle} of the link $K$.
\end{definition}

It can be proved that the fundamental quandle of the link does not depend on the diagram of the link. So the fundamental quandle $\Gamma_K$ is an invariant of the link $K$.

\subsection{Quandle colourings invariant}\label{Subsection:QuandleColourings} Here we will recall definition of the classical colourings counting invariant for oriented links.

\begin{definition}
    Let $\mathcal{X} = (X, *)$ be a finite quandle, and let $K$ be an oriented knot or link with a diagram $D$. The map $\xi\colon \mathcal{A}(D)\to X$ is called \emph{colouring} of the diagram $D$ if in the neighbourhood of each crossing of the diagram $D$: $\xi(a_1) * \xi(a_2) = \xi(a_3)$, where $a_1, a_2, a_3$ are arc, shown on the figure \ref{Figure:Crossings}.
\end{definition}

Denote the set of all colourings of the diagram $D$ by $C_{\mathcal{X}}(D)$.

Each colouring is a homomorphism from the fundamental quandle $\Gamma_{K}$ to the finite quandle $\mathcal{X}$. Denote the set of all these homomorphisms by $C_{\mathcal{X}}(K)$. Thus, first, the cardinality of the set $|C_{\mathcal{X}}(K)|$ is a link invariant, and second, there is a natural bijection between the sets $C_{\mathcal{X}}(D_1)$ and $C_{\mathcal{X}}(D_2)$ for any two diagrams $D_1, D_2$ of the same link $K$.

If $\mathcal{T} = (X, *)$ is the finite trivial quandle, then $|C_{\mathcal{T}}(K)| = |K|^{|X|}$, where $|K|$ is the number of components of the link $K$. If $\mathcal{G} = (G, *)$ is the conjugation quandle for the finite group $G$, then $|C_{\mathcal{G}}(K)|$ is equal to the number of homomorphisms from the knot group $\pi_{1}(K)$ to the group $G$.

\subsection{Hierarchical quandles}\label{Subsection:HierarchicalQuandle} In this small subsection we will give a definition of the hierarchical quandle.

\begin{definition}
    Let $\mathcal{X} = (X, *)$ be a finite quandle, $X = \{x_1, \ldots, x_n\}$. The \emph{hierarchical quandle} over the quandle $\mathcal{X}$ is a set $Y$ with a family of binary operations $*_{x_1}^{x_2}\colon Y\times Y\to Y$, indexed by all pairs $(x_1, x_2)\in X\times X$, satisfying the following conditions:
    \begin{enumerate}
        \item For any $x\in X$ and any $y\in Y$: $y *_{x}^{x} y = y$;
        \item For any $x_1, x_2\in X$ and any $a, b\in Y$ there exists unique $y\in Y$ such that $y *_{x_1}^{x_2} a = b$;
        \item For any $x_1, x_2, x_3\in X$ and any $y_1, y_2, y_3\in Y$: $$(y_1 *_{x_1}^{x_2} y_2) *_{x_1 *x_2}^{x_3} y_3 = (y_1 *_{x_1}^{x_3} * y_3) *_{x_1 * x_3}^{x_2 * x_3} (y_2 *_{x_2}^{x_3} y_3).$$
    \end{enumerate}
\end{definition}

It follows from the definition that for any $x\in X$: $(Y, *_{x}^{x})$ is a quandle.

The unique solution of the equation from the second axiom denote by $y = b /_{x_1}^{x_2} a$;

\subsection{Hierarchical quandle colourings invariant}\label{Subsection:HierarchicalQuandleColourings} Let $\mathcal{X} = (X, *)$ be a finite quandle. By \emph{coloured oriented link} we mean the pair $(K, \varphi)$, where $\varphi\colon \Gamma_{K} \to X$ is a homomorphism. If the link $K$ is defined by the diagram $D$, then the homomorphism $\varphi$ is defined by the colouring $\xi\colon\mathcal{A}(D)\to X$. We will say that the pair $(D, \xi)$ is a \emph{diagram} of the coloured link $(K, \varphi)$.

\begin{definition}
    Let $\mathcal{X} = (X, *)$ be a finite quandle, $(K, \varphi)$ an oriented link coloured by $\mathcal{X}$, $(D, \xi)$ a diagram of $(K, \varphi)$, and let $\mathcal{Y} = (Y, \{*_{x_1}^{x_2}\}_{x_1, x_2\in X})$ be a finite hierarchical quandle over $\mathcal{X}$. The \emph{colouring} of the diagram $(D, \xi)$ by the hierarchical quandle $\mathcal{Y}$ is a map $\zeta\colon \mathcal{A}(D)\to Y$ such that in the neighbourhood of each crossing of the diagram $D$: $\zeta(a_3) = \zeta(a_1) *_{x_1}^{x_2} \zeta(a_2)$, where $a_1, a_2, a_3$ are arcs in the neighbourhood of the crossing (see figure \ref{Figure:Crossings}), and $x_1 = \xi(a_1), x_2 = \xi(a_2)$.
\end{definition}

\begin{remark}
    For any colouring $\zeta$ of the coloured diagram $(D, \xi)$ in the neighbourhood of the crossing we have two colourings of the diagram arcs: colours from the quandle $\mathcal{X}$ and colours from the hierarchical quandle $\mathcal{Y}$. To distinguish these colours on figures, we draw the quandle colours in rectangles (see figure \ref{Figure:HierarchicalColours})
\end{remark}

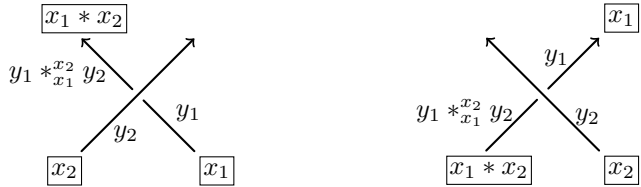
\begin{figure}[h]
    \begin{center}
        \begin{tikzpicture}[scale=0.75]
            \draw[knot_arrow, knot_diagram] (-1, -1) -- (1, 1) node[pos=0.15, right, xshift=2.0] {$y_2$};
            \draw (-1.6, -1.6) rectangle ++(0.6, 0.5) node[pos=.5] {$x_2$};

            \draw[knot_diagram] (1, -1) -- (0.1, -0.1) node[above right, pos=0.4, xshift=-3.0] {$y_1$};
            \draw (1.1, -1.6) rectangle ++(0.6, 0.5) node[pos=.5] {$x_1$};

            \draw[knot_arrow, knot_diagram] (-0.1, 0.1) -- (-1, 1) node[pos=0.3, left] {$y_1 *_{x_1}^{x_2} y_2$};
            \draw (-1.7, 1.1) rectangle ++(1.5, 0.5) node[pos=.5] {$x_1 * x_2$};
        \end{tikzpicture}
        \hspace{2cm}
        \begin{tikzpicture}[scale=0.75]
            \draw[knot_arrow, knot_diagram] (1, -1) -- (-1, 1) node[pos=0.3, right] {$y_2$};
            \draw (1.1, -1.6) rectangle ++(0.6, 0.5) node[pos=.5] {$x_2$};

            \draw[knot_diagram] (-1, -1) -- (-0.1, -0.1) node[pos=0.75, left] {$y_1 *_{x_1}^{x_2} y_2$};
            \draw (-1.7, -1.6) rectangle ++(1.5, 0.5) node[pos=.5] {$x_1 * x_2$};

            \draw[knot_arrow, knot_diagram] (0.1, 0.1) -- (1, 1) node[pos=0.6, left] {$y_1$};
            \draw (1.1, 1.1) rectangle ++(0.6, 0.5) node[pos=.5] {$x_1$};
        \end{tikzpicture}
    \end{center}
    \caption{\label{Figure:HierarchicalColours}Colours of arcs in the neighbourhood of crossings}
\end{figure}

Denote by $C_{\mathcal{Y}}(D, \xi)$ the set of all colourings of the diagram $(D, \xi)$ by the hierarchical quandle $\mathcal{Y}$.

\begin{theorem}
    \label{Theorem:HeirarchicalColourings}
    Let $\mathcal{X} = (X, *)$ be a finite quandle, $(K, \varphi)$ be an oriented link coloured by $\mathcal{X}$, and let $\mathcal{Y} = (Y, \{*_{x_1}^{x_2}\}_{x_1, x_2\in X})$ be a finite hierarchical quandle over $\mathcal{X}$. Then for any two diagrams $D_1, D_2$ of the link $K$ with colourings $\xi_1, \xi_2$ there is a natural bijection between the sets $C_{\mathcal{Y}}(D_1, \xi_1)$ and $C_{\mathcal{Y}}(D_2, \xi_2)$.
\end{theorem}
\begin{proof}
    The proof of the theorem is typical and straightforward. It's enough to check the theorem statement only for the cases where the diagram $D_2$ is obtained from the diagram $D_1$ by one of the five Reidemeister moves shown on the figures \ref{Figure:R1}, \ref{Figure:R2} and \ref{Figure:R3} (\cite{Po}). It's easy to check that in all these cases, for any colouring of the diagram before the move, there is only one colouring after the move that matches the original colouring outside the move area. For the first Reidemeister moves this follows from the first axiom of the hierarchical quandle, for the second moves --- from the second axiom and for the third Reidemeister move --- from the third axiom.
\end{proof}

\begin{figure}[h]
    \begin{tikzpicture}[scale=1.0, use Hobby shortcut, baseline={([yshift=-1.5ex]current bounding box.center)}]
        \begin{knot}[
            clip width=5,
            consider self intersections=true,
            ignore endpoint intersections=false,
            flip crossing/.list={1}
        ]
            \strand (0, -1) .. (0, -0.25) .. (0.375, 0.35) .. (0.75, 0) .. (0.375, -0.35) .. (0, 0.25) .. (0, 0.5);
        \end{knot}
        \draw[knot_diagram, knot_arrow] (0, 0.5) -- (0, 1.0);

        \draw (-0.15, -1.4) rectangle ++(0.3, 0.3) node[pos=0.5] {$x$};
        \draw (-0.15, 1.1) rectangle ++(0.3, 0.3) node[pos=0.5] {$x$};
        \draw (0, -0.5) node[left] {$y$};
        \draw (0, 0.5) node[left] {$y$};
    \end{tikzpicture}
    $\longleftrightarrow$
    \begin{tikzpicture}[scale=1.0, baseline={([yshift=0.0ex]current bounding box.center)}]
        \draw[knot_diagram, knot_arrow] (0, -1) -- (0, 1) node[pos=0.5, right] {$y$};
        \draw (-0.15, -1.4) rectangle ++(0.3, 0.3) node[pos=0.5] {$x$};
    \end{tikzpicture}
    $\longleftrightarrow$
    \begin{tikzpicture}[scale=1.0, use Hobby shortcut, baseline={([yshift=-1.5ex]current bounding box.center)}]
        \begin{knot}[
            clip width=5,
            consider self intersections=true,
            ignore endpoint intersections=false,
        ]
            \strand (0, -1) .. (0, -0.25) .. (0.375, 0.35) .. (0.75, 0) .. (0.375, -0.35) .. (0, 0.25) .. (0, 0.5);
        \end{knot}
        \draw[knot_diagram, knot_arrow] (0, 0.5) -- (0, 1.0);

        \draw (-0.15, -1.4) rectangle ++(0.3, 0.3) node[pos=0.5] {$x$};
        \draw (-0.15, 1.1) rectangle ++(0.3, 0.3) node[pos=0.5] {$x$};
        \draw (0, -0.5) node[left] {$y$};
        \draw (0, 0.5) node[left] {$y$};
    \end{tikzpicture}
    \caption{\label{Figure:R1}Colourings before and after of two versions of the first Reidemeister moves}
\end{figure}
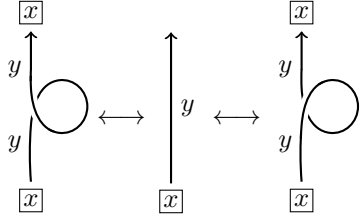

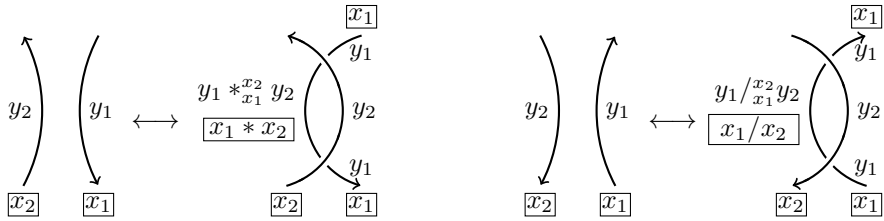
\begin{figure}[h]
    \begin{tikzpicture}[use Hobby shortcut, baseline={([yshift=0.0ex]current bounding box.center)}]
        \begin{knot}
            \strand[knot_arrow] (-0.5, -1) .. (-0.25, 0) .. (-0.5, 1);
            \strand[knot_arrow_rev] (0.5, -1) .. (0.25, 0) .. (0.5, 1);
        \end{knot}
        \draw (-0.7, -1.4) rectangle ++(0.4, 0.3) node[pos=0.5] {$x_2$};
        \draw (0.3, -1.4) rectangle ++(0.4, 0.3) node[pos=0.5] {$x_1$};
        \draw (-0.25, 0.0) node[left] {$y_2$};
        \draw (0.25, 0.0) node[right] {$y_1$};
    \end{tikzpicture}
    $\longleftrightarrow$
    \begin{tikzpicture}[use Hobby shortcut, baseline={([yshift=-1.5ex]current bounding box.center)}]
        \begin{knot}[
            clip width=5,
            ignore endpoint intersections=false,
        ]
            \strand[knot_arrow] (-0.5, -1.0) .. (0.25, 0.0) .. (-0.5, 1.0);
            \strand[knot_arrow_rev] (0.5, -1.0) .. (-0.25, 0.0) .. (0.5, 1.0);
        \end{knot}
        \draw (-0.7, -1.4) rectangle ++(0.4, 0.3) node[pos=0.5] {$x_2$};
        \draw (0.3, -1.4) rectangle ++(0.4, 0.3) node[pos=0.5] {$x_1$};
        \draw (0.3, 1.1) rectangle ++(0.4, 0.3) node[pos=0.5] {$x_1$};
        \draw (-0.25, 0.25) node[left] {$y_1 *_{x_1}^{x_2} y_2$};
        \draw (-1.6, -0.4) rectangle ++(1.2, 0.3) node[pos=0.5] {$x_1 * x_2$};
        \draw (0.5, 1.0) node[below] {$y_1$};
        \draw (0.5, -1.0) node[above] {$y_1$};
        \draw (0.25, 0.0) node[right] {$y_2$};
    \end{tikzpicture}
    \hspace{1.5cm}
    \begin{tikzpicture}[use Hobby shortcut, baseline={([yshift=0.0ex]current bounding box.center)}]
        \begin{knot}
            \strand[knot_arrow_rev] (-0.5, -1) .. (-0.25, 0) .. (-0.5, 1);
            \strand[knot_arrow] (0.5, -1) .. (0.25, 0) .. (0.5, 1);
        \end{knot}
        \draw (-0.7, -1.4) rectangle ++(0.4, 0.3) node[pos=0.5] {$x_2$};
        \draw (0.3, -1.4) rectangle ++(0.4, 0.3) node[pos=0.5] {$x_1$};
        \draw (-0.25, 0.0) node[left] {$y_2$};
        \draw (0.25, 0.0) node[right] {$y_1$};
    \end{tikzpicture}
    $\longleftrightarrow$
    \begin{tikzpicture}[use Hobby shortcut, baseline={([yshift=-1.5ex]current bounding box.center)}]
        \begin{knot}[
            clip width=5,
            ignore endpoint intersections=false,
        ]
            \strand[knot_arrow_rev] (-0.5, -1.0) .. (0.25, 0.0) .. (-0.5, 1.0);
            \strand[knot_arrow] (0.5, -1.0) .. (-0.25, 0.0) .. (0.5, 1.0);
        \end{knot}
        \draw (-0.7, -1.4) rectangle ++(0.4, 0.3) node[pos=0.5] {$x_2$};
        \draw (0.3, -1.4) rectangle ++(0.4, 0.3) node[pos=0.5] {$x_1$};
        \draw (0.3, 1.1) rectangle ++(0.4, 0.3) node[pos=0.5] {$x_1$};
        \draw (-0.25, 0.25) node[left] {$y_1 /_{x_1}^{x_2} y_2$};
        \draw (-1.6, -0.45) rectangle ++(1.2, 0.4) node[pos=0.5] {$x_1 / x_2$};
        \draw (0.5, 1.0) node[below] {$y_1$};
        \draw (0.5, -1.0) node[above] {$y_1$};
        \draw (0.25, 0.0) node[right] {$y_2$};
    \end{tikzpicture}
    \caption{\label{Figure:R2}Colouring before and after two versions of the second Reidemeister move}
\end{figure}

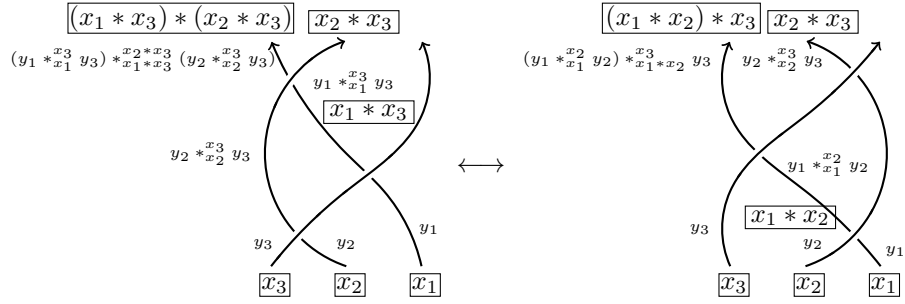
\begin{figure}[h]
    \begin{center}
        \begin{tikzpicture}[use Hobby shortcut, baseline={([yshift=-2.0ex]current bounding box.center)}]
            \begin{knot}[
                clip width=5,
            ]
                \strand[knot_arrow] (0, 0) .. (1, 1) .. (2, 2) .. (2, 3);
                \strand[knot_arrow] (1, 0) .. (0, 1) .. (0, 2) .. (1, 3);
                \strand[knot_arrow] (2, 0) .. (1.5, 1) .. (1, 1.5) .. (0, 3);
            \end{knot}
            \draw (-0.15, -0.4) rectangle ++(0.4, 0.3) node[pos=0.5] {$x_3$};
            \draw (0.85, -0.4) rectangle ++(0.4, 0.3) node[pos=0.5] {$x_2$};
            \draw (1.85, -0.4) rectangle ++(0.4, 0.3) node[pos=0.5] {$x_1$};

            \draw (0.5, 3.1) rectangle ++(1.2, 0.3) node[pos=0.5] {$x_2 * x_3$};
            \draw (-2.7, 3.1) rectangle ++(3.0, 0.4) node[pos=0.5] {$(x_1 * x_3) * (x_2 * x_3)$};

            \draw (2.1, 0.5) node {{\tiny$y_1$}};
            \draw (1.0, 0.3) node {{\tiny$y_2$}};
            \draw (-0.1, 0.3) node {{\tiny$y_3$}};

            \draw (1.1, 2.45) node {{\tiny$y_1 *_{x_1}^{x_3} y_3$}};
            \draw (0.7, 1.9) rectangle ++(1.2, 0.3) node[pos=0.5] {$x_1 * x_3$};

            \draw (-0.8, 1.5) node {{\tiny$y_2 *_{x_2}^{x_3} y_3$}};

            \draw (-1.7, 2.75) node {{\tiny $(y_1 *_{x_1}^{x_3} y_3) *_{x_1 * x_3}^{x_2 * x_3} (y_2 *_{x_2}^{x_3} y_3)$}};
        \end{tikzpicture}
        $\longleftrightarrow$
        \begin{tikzpicture}[use Hobby shortcut, baseline={([yshift=-2.0ex]current bounding box.center)}]
            \begin{knot}[
                clip width=5,
            ]
                \strand[knot_arrow] (0, 0) .. (0, 1) .. (1, 2) .. (2, 3);
                \strand[knot_arrow] (1, 0) .. (2, 1) .. (2, 2) .. (1, 3);
                \strand[knot_arrow] (2, 0) .. (1, 1) .. (0, 2) .. (0, 3);
            \end{knot}
            \draw (-0.15, -0.4) rectangle ++(0.4, 0.3) node[pos=0.5] {$x_3$};
            \draw (0.85, -0.4) rectangle ++(0.4, 0.3) node[pos=0.5] {$x_2$};
            \draw (1.85, -0.4) rectangle ++(0.4, 0.3) node[pos=0.5] {$x_1$};

            \draw (0.5, 3.1) rectangle ++(1.2, 0.3) node[pos=0.5] {$x_2 * x_3$};
            \draw (-1.7, 3.1) rectangle ++(2.1, 0.4) node[pos=0.5] {$(x_1 * x_2) * x_3$};

            \draw (2.2, 0.2) node {{\tiny$y_1$}};
            \draw (1.1, 0.3) node {{\tiny$y_2$}};
            \draw (-0.4, 0.5) node {{\tiny$y_3$}};

            \draw (1.3, 1.35) node {{\tiny$y_1 *_{x_1}^{x_2} y_2$}};
            \draw (0.2, 0.5) rectangle ++(1.2, 0.3) node[pos=0.5] {$x_1 * x_2$};

            \draw (0.7, 2.75) node {{\tiny$y_2 *_{x_2}^{x_3} y_3$}};

            \draw (-1.5, 2.75) node {{\tiny$(y_1 *_{x_1}^{x_2} y_2) *_{x_1 * x_2}^{x_3} y_3$}};
        \end{tikzpicture}
    \end{center}
    \caption{\label{Figure:R3}Colouring before and after the third Reidemeister move}
\end{figure}

It follows from the theorem that the cardinality of the set $C_{\mathcal{Y}}(D, \xi)$ does not depend on the diagram $(D, \xi)$ of the coloured oriented link $(K, \varphi)$. So this value is an invariant of the coloured link $(K, \varphi)$. Furthermore, the multi-set $\{|C_{\mathcal{Y}}(K, \varphi)|, \varphi\in C_{\mathcal{X}(K)}\}$ for all different colourings of the link $K$ is an invariant of the link $K$.

\section{Cocycle invariants}\label{Section:CocycleInvariants}

Throughout this section we will fix the finite quandle $\mathcal{X} = (X, *)$, the finite hierarchical quandle $\mathcal{Y} = (Y, \{*_{x_1}^{x_2}\}_{x_1, x_2\in X})$ over $\mathcal{X}$, and the associative ring $\mathbb{K}$ with unit. The main goal of this section is to construct a cocycle invariant for coloured links similar to the classical cocycle invariants from \cite{CJKLS}.

\subsection{Chain groups.}\label{Subsection:ChainGroups} For any $x_1, \ldots, x_n\in X$ denote $V_{n}^{x_1, \ldots, x_n}(\mathcal{X}, \mathcal{Y})$ the free module over $\mathbb{K}$ generated by $n$-tuples $(y_1, \ldots, y_n)\in Y^{n}$. All modules $V_n^{x_1, \ldots, x_n}(\mathcal{X}, \mathcal{Y})$ are isomorphic. The labels $x_1, \ldots, x_n$ are needed to distinguish these modules. Define $$V_n(\mathcal{X}, \mathcal{Y}) = \bigoplus_{x_1, \ldots, x_n \in X}V^{x_1, \ldots, x_n}_n(\mathcal{X}, \mathcal{Y}).$$

It's clear that $\dim V^{x_1, \ldots, x_n}_n(\mathcal{X}, \mathcal{Y}) = |Y|^n$ for any $x_1, \ldots, x_n\in X$, and hence $\dim V_n(\mathcal{X}, \mathcal{Y}) = |X|^n\cdot |Y|^n$. If $n\leqslant 0$, postulate $V_n (\mathcal{X}, \mathcal{Y}) = 0$.

Denote the standard basis of the module $V_n(\mathcal{X}, \mathcal{Y})$ by $$\{(y_1, \ldots, y_n)^{x_1, \ldots, x_n} | y_1, \ldots, y_n\in Y, x_1, \ldots, x_n\in X\}.$$

\subsection{Boundary maps}\label{Subsection:BoundaryMaps} Let $n\geqslant 1$. For each $i\in \{1, \ldots, n\}$ define two maps $\lambda_{n, i}, \rho_{n, i}\colon V_n(\mathcal{X}, \mathcal{Y})\to V_{n - 1}(\mathcal{X}, \mathcal{Y})$ on the standard basis of $V_n(\mathcal{X}, \mathcal{Y})$ by formulas
\begin{multline*}
    \lambda_{n, i}((y_1, \ldots, y_n)^{x_1, \ldots, x_n}) = \\ = (y_1 *_{x_1}^{x_i} y_i, \ldots, y_{i - 1} *_{x_{i - 1}}^{x_i} y_i, y_{i + 1}, \ldots, y_{n})^{x_1 * x_i, \ldots, x_{i - 1} * x_i, x_{i + 1}, \ldots, x_n},
\end{multline*}
$$
    \rho_{n, i}((y_1, \ldots, y_n)^{x_1, \ldots, x_n}) = (y_1, \ldots, y_{i - 1}, y_{i + 1}, \ldots, y_n)^{x_1, \ldots, x_{i - 1}, x_{i + 1}, \ldots, x_n}.
$$

Denote
\begin{center}
    $l_n = \sum\limits_{i = 1}^{n}(-1)^{i + 1}\lambda_{n, i}$ and $r_n = \sum\limits_{i = 1}^{n}(-1)^{i + 1}\rho_{n, i}$.
\end{center}

Postulate that $l_n = r_n = 0$ for $n\leqslant 0$.

\begin{example}
    \begin{multline*}
        \partial_3((y_1, y_2, y_3)^{x_1, x_2, x_3}) = -(y_1 *_{x_1}^{x_2} y_2, y_3)^{x_1 * x_2, x_3} + \\ + (y_1 *_{x_1}^{x_3} y_3, y_2 *_{x_2}^{x_3} y_3)^{x_1 * x_3, x_2 * x_3} + (y_1, y_3)^{x_1, x_3} - (y_1, y_2)^{x_1, x_2}.
    \end{multline*}
\end{example}

\begin{remark}
    The maps $\lambda_{n, i}$ and $\rho_{n, i}$ have a nice description using tangle diagrams. This approach comes from \cite{Pr} and is used in \cite{K2, K3}.
    
    Consider the tangle $L_{n, i}$ shown on the figure \ref{Figure:TangleL}. If we colour the top points of the tangle $L_{n, i}$ by the quandle colours $x_1, \ldots, x_n$ and by the hierarchical quandle colours $y_1, \ldots, y_n$, then the image $\lambda_{n, i}((y_1, \ldots, y_n)^{x_1, \ldots, x_n})$ is exactly the quandle colours $$x_1 * x_i, \ldots, x_{i - 1} * x_i, x_{i + 1}, \ldots, x_n$$ and hierarchical quandle colours $$y_1 *_{x_1}^{x_i} y_i, \ldots, y_{i - 1} *_{x_{i - 1}}^{x_i} y_i, y_{i + 1}, \ldots, y_{n}$$ of the bottom points of the tangle $L_{n, i}$.

    The map $\rho_{n, i}$ can be described in a similar way. The only difference is the use of the tangle $R_{n, i}$ shown on the figure \ref{Figure:TangleR}.
\end{remark}

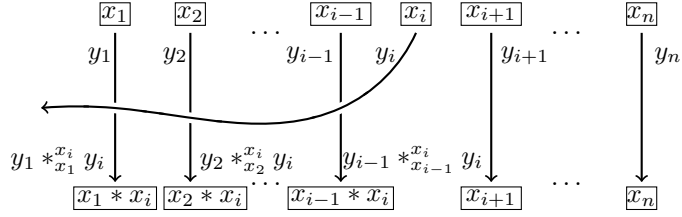
\begin{figure}[h]
    \begin{center}
        \begin{tikzpicture}[use Hobby shortcut]
            \begin{knot}[
                clip width=5,
            ]
                \strand[knot_arrow] (4, 2) .. (3, 1) .. (0, 1) .. (-1, 1);

                \strand[knot_arrow] (0, 2) -- (0, 0);
                \strand[knot_arrow] (1, 2) -- (1, 0);
                \strand[knot_arrow] (3, 2) -- (3, 0);

                \strand[knot_arrow] (5, 2) -- (5, 0);
                \strand[knot_arrow] (7, 2) -- (7, 0);
            \end{knot}
            \draw (2, 2) node {$\ldots$};
            \draw (2, 0) node {$\ldots$};
            \draw (6, 2) node {$\ldots$};
            \draw (6, 0) node {$\ldots$};

            \draw (-0.2, 2.1) rectangle ++(0.4, 0.3) node[pos=0.5] {$x_1$};
            \draw (0.8, 2.1) rectangle ++(0.4, 0.3) node[pos=0.5] {$x_2$};
            \draw (2.6, 2.1) rectangle ++(0.8, 0.3) node[pos=0.5] {$x_{i - 1}$};
            \draw (3.8, 2.1) rectangle ++(0.4, 0.3) node[pos=0.5] {$x_i$};
            \draw (4.6, 2.1) rectangle ++(0.8, 0.3) node[pos=0.5] {$x_{i + 1}$};
            \draw (6.8, 2.1) rectangle ++(0.4, 0.3) node[pos=0.5] {$x_n$};

            \draw (-0.55, -0.35) rectangle ++(1.1, 0.3) node[pos=0.5] {$x_1 * x_i$};
            \draw (0.65, -0.35) rectangle ++(1.1, 0.3) node[pos=0.5] {$x_2 * x_i$};
            \draw (2.3, -0.35) rectangle ++(1.4, 0.3) node[pos=0.5] {$x_{i - 1} * x_i$};
            \draw (4.6, -0.35) rectangle ++(0.8, 0.3) node[pos=0.5] {$x_{i + 1}$};
            \draw (6.8, -0.35) rectangle ++(0.4, 0.3) node[pos=0.5] {$x_n$};

            \draw (-0.2, 1.7) node {$y_1$};
            \draw (0.8, 1.7) node {$y_2$};
            \draw (2.6, 1.7) node {$y_{i - 1}$};
            \draw (3.6, 1.7) node {$y_i$};
            \draw (5.45, 1.7) node {$y_{i + 1}$};
            \draw (7.35, 1.7) node {$y_n$};

            \draw (0.0, 0.3) node[left] {$y_1 *_{x_1}^{x_i} y_i$};
            \draw (1.0, 0.3) node[right] {$y_2 *_{x_2}^{x_i} y_i$};
            \draw (2.9, 0.3) node[right] {$y_{i - 1} *_{x_{i - 1}}^{x_i} y_i$};
        \end{tikzpicture}
    \end{center}
    \caption{\label{Figure:TangleL}Colouring of the tangle $L_{n, i}$}
\end{figure}

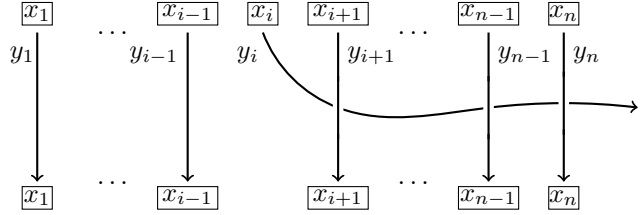
\begin{figure}[h]
    \begin{center}
        \begin{tikzpicture}[use Hobby shortcut]
            \begin{knot}[
                clip width=5,
            ]
                \strand[knot_arrow] (0, 2) -- (0, 0);
                \strand[knot_arrow] (2, 2) -- (2, 0);

                \strand[knot_arrow] (4, 2) -- (4, 0);
                \strand[knot_arrow] (6, 2) -- (6, 0);
                \strand[knot_arrow] (7, 2) -- (7, 0);

                \strand[knot_arrow] (3, 2) .. (4, 1) .. (6, 1) .. (8, 1);
            \end{knot}
            \draw (1, 2) node {$\ldots$};
            \draw (1, 0) node {$\ldots$};
            \draw (5, 2) node {$\ldots$};
            \draw (5, 0) node {$\ldots$};

            \draw (-0.2, 2.1) rectangle ++(0.4, 0.3) node[pos=0.5] {$x_1$};
            \draw (1.6, 2.1) rectangle ++(0.8, 0.3) node[pos=0.5] {$x_{i - 1}$};
            \draw (2.8, 2.1) rectangle ++(0.4, 0.3) node[pos=0.5] {$x_i$};
            \draw (3.6, 2.1) rectangle ++(0.8, 0.3) node[pos=0.5] {$x_{i + 1}$};
            \draw (5.6, 2.1) rectangle ++(0.8, 0.3) node[pos=0.5] {$x_{n - 1}$};
            \draw (6.8, 2.1) rectangle ++(0.4, 0.3) node[pos=0.5] {$x_n$};

            \draw (-0.2, -0.35) rectangle ++(0.4, 0.3) node[pos=0.5] {$x_1$};
            \draw (1.6, -0.35) rectangle ++(0.8, 0.3) node[pos=0.5] {$x_{i - 1}$};
            \draw (3.6, -0.35) rectangle ++(0.8, 0.3) node[pos=0.5] {$x_{i + 1}$};
            \draw (5.6, -0.35) rectangle ++(0.8, 0.3) node[pos=0.5] {$x_{n - 1}$};
            \draw (6.8, -0.35) rectangle ++(0.4, 0.3) node[pos=0.5] {$x_n$};

            \draw (-0.2, 1.7) node {$y_1$};
            \draw (2.0, 1.7) node[left] {$y_{i - 1}$};
            \draw (2.8, 1.7) node {$y_i$};
            \draw (4.0, 1.7) node[right] {$y_{i + 1}$};
            \draw (6.0, 1.7) node[right] {$y_{n - 1}$};
            \draw (7.0, 1.7) node[right] {$y_n$};
        \end{tikzpicture}
    \end{center}
    \caption{\label{Figure:TangleR}Colouring of the tangle $R_{n, i}$}
\end{figure}

\begin{lemma}
    \label{Lemma:LR}
    For all $n\geqslant 1$ the following identities holds:
    \begin{enumerate}
        \item $l_{n - 1}\circ l_{n} = 0$;
        \item $r_{n - 1}\circ r_{n} = 0$;
        \item $l_{n - 1}\circ r_{n} + r_{n - 1}\circ l_{n} = 0$.
    \end{enumerate}
\end{lemma}
\begin{proof}
    All three statements of the lemma can be proved in a similar way by using a diagrammatic approach to the maps $\lambda_{n, i}$ and $\rho_{n, i}$. Let's prove the first statement.

    Note that if $j > i$ then $$(\lambda_{n, j - 1}\circ \lambda_{n, i})((y_1, \ldots, y_n)^{x_1, \ldots, x_n}) = (\lambda_{n, i}\circ \lambda_{n, j})((y_1, \ldots, y_n)^{x_1, \ldots, x_n}).$$
    This equality follows from the fact that with fixed colouring of the top points of the tangles shown on the figure \ref{Figure:TanglesIJ} (for the composition $\lambda_{n, j - 1}\circ \lambda_{n, i}$ on the left and for the composition $\lambda_{n, i}\circ \lambda_{n, j}$ on the right) the colours of the bottom points are the same.

    \begin{figure}[h]
        \begin{tikzpicture}[scale=0.6, use Hobby shortcut]
            \begin{knot}[
                clip width=5,
                consider self intersections=true,
		        ignore endpoint intersections=false,
            ]
                \strand[knot_arrow] (3, 3) .. (2, 2) .. (0, 2) .. (-1, 2);
                \strand[knot_arrow] (5, 3) .. (3, 1) .. (0, 1) .. (-1, 1);
                \strand[knot_arrow] (0, 3) .. (0, 0);
                \strand[knot_arrow] (2, 3) .. (2, 0);
                \strand[knot_arrow] (6, 3) .. (6, 0);
                \strand[knot_arrow] (8, 3) .. (8, 0);
            \end{knot}
            \draw (1, 3) node {$\ldots$};
            \draw (4, 3) node {$\ldots$};
            \draw (7, 3) node {$\ldots$};

            \draw (1, 0) node {$\ldots$};
            \draw (4, 0) node {$\ldots$};
            \draw (7, 0) node {$\ldots$};

            \draw (0, 3) node[above] {$1$};
            \draw (3, 3) node[above] {$i$};
            \draw (5, 3) node[above] {$j$};
            \draw (8, 3) node[above] {$n$};
        \end{tikzpicture}
        \hfill
        \begin{tikzpicture}[scale=0.6, use Hobby shortcut]
            \begin{knot}[
                clip width=5,
                consider self intersections=true,
		        ignore endpoint intersections=false,
            ]
                \strand[knot_arrow] (5, 3) .. (3, 2) .. (0, 2) .. (-1, 2);
                \strand[knot_arrow] (3, 3) .. (2, 1) .. (0, 1) .. (-1, 1);
                \strand[knot_arrow] (0, 3) .. (0, 0);
                \strand[knot_arrow] (2, 3) .. (2, 0);
                \strand[knot_arrow] (6, 3) .. (6, 0);
                \strand[knot_arrow] (8, 3) .. (8, 0);
            \end{knot}
            \draw (1, 3) node {$\ldots$};
            \draw (4, 3) node {$\ldots$};
            \draw (7, 3) node {$\ldots$};

            \draw (1, 0) node {$\ldots$};
            \draw (4, 0) node {$\ldots$};
            \draw (7, 0) node {$\ldots$};

            \draw (0, 3) node[above] {$1$};
            \draw (3, 3) node[above] {$i$};
            \draw (5, 3) node[above] {$j$};
            \draw (8, 3) node[above] {$n$};
        \end{tikzpicture}
        \caption{\label{Figure:TanglesIJ}Diagrams of the composition $\lambda_{n, j - 1}\circ \lambda_{n, i}$ (on the left) and $\lambda_{n, i}\circ \lambda_{n, j}$ (on the right)}
    \end{figure}
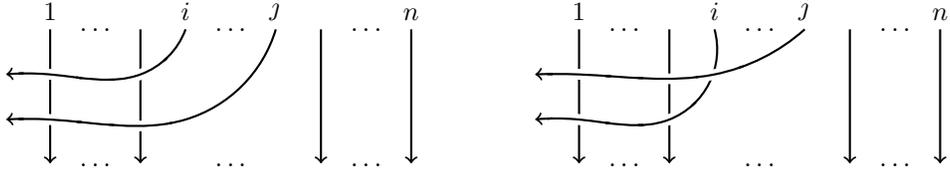

    In the expression $(l_{n - 1}\circ l_n)((y_1, \ldots, y_n)^{x_1, \ldots, x_n})$ as a sum these values $(\lambda_{n, j - 1}\circ \lambda_{n, i})((y_1, \ldots, y_n)^{x_1, \ldots, x_n})$ and $(\lambda_{n, i}\circ \lambda_{n, j})((y_1, \ldots, y_n)^{x_1, \ldots, x_n})$ comes with opposite signs. So, the total sum is zero.

    The second and third statements are similar.
\end{proof}

For any $n\in\mathbb{Z}$ define the homomorphism $\partial_n\colon V_n(\mathcal{X}, \mathcal{Y})\to V_{n - 1}(\mathcal{X}, \mathcal{Y})$ by the formula $\partial_{n} = l_n - r_n$.

\begin{lemma}
    \label{Lemma:Boundary}
    $\partial_{n - 1} \circ \partial_n = 0$.
\end{lemma}
\begin{proof}
    The lemma statement follows from the lemma \ref{Lemma:LR}. Indeed
    \begin{multline*}
        \partial_{n - 1} \circ \partial_n = (l_n - r_n)\circ (l_{n - 1} - r_{n - 1}) = \\ = l_n\circ r_n - l_n\circ r_{n - 1} - r_n\circ l_{n - 1} + r_n\circ r_{n - 1} = 0.
    \end{multline*}
\end{proof}

So all homomorphisms $\partial_n$ are boundary maps.

\subsection{Chain complex}\label{Subsection:ChainComplex} In each module $V_n^{x_1, \ldots, x_n}(\mathcal{X}, \mathcal{Y})$, $n\geqslant 1$, $x_1, \ldots, x_n\in X$, consider the submodule $W_n^{x_1, \ldots, x_n}(\mathcal{X}, \mathcal{Y})$, generated by those elements $(y_1, \ldots, y_n)^{x_1, \ldots, x_n}$ such that there exists an index $i\in\{1, \ldots, n - 1\}$ such that $y_i = y_{i + 1}$ and $x_i = x_{i + 1}$.

\begin{remark}
    If there are no two equal consequence elements among the elements $x_1, \ldots, x_n$, then the submodule $W_{n}^{x_1, \ldots, x_n}(\mathcal{X}, \mathcal{Y})$ is trivial. Also all submodules $W_1^x(\mathcal{X}, \mathcal{Y})$, $x\in X$ are trivial.
\end{remark}

\begin{example}
    The submodule $W_{3}^{x, x, x'}(\mathcal{X}, \mathcal{Y})$ generated by elements of the form $(y, y, y')^{x, x, x'}$, $y, y'\in Y$, and the submodule $W_{3}^{x', x, x}(\mathcal{X}, \mathcal{Y})$ generated by elements of the form $(y', y, y)^{x', x, x}$, $y, y'\in Y$.
\end{example}

Denote $$W_n(\mathcal{X}, \mathcal{Y}) = \bigoplus_{x_1, \ldots, x_n\in X}W_n^{x_1, \ldots, x_n}(\mathcal{X}, \mathcal{Y}).$$

\begin{lemma}
    \label{Lemma:W}
    For all $n\geqslant 1$:
    \begin{enumerate}
        \item $l_n(W_n(\mathcal{X}, \mathcal{Y}))\subseteq W_{n - 1}(\mathcal{X}, \mathcal{Y})$;
        \item $r_n(W_n(\mathcal{X}, \mathcal{Y}))\subseteq W_{n - 1}(\mathcal{X}, \mathcal{Y})$.
    \end{enumerate}
\end{lemma}
\begin{proof}
    Prove the first statement, the second is similar.

    Let $(y_1, \ldots, y_n)^{x_1, \ldots, x_n}\in W_n(\mathcal{X}, \mathcal{Y})$, and let $y_i = y_{i + 1}$ and $x_i = x_{i + 1}$ for some $i = 1, \ldots, n - 1$. Then for any $j \neq i, i + 1$: $$\lambda_{n, j}((y_1, \ldots, y_n)^{x_1, \ldots, x_n}) \in W_{n - 1}(\mathcal{X}, \mathcal{Y}).$$ Note that $$\lambda_{n, i}((y_1, \ldots, y_n)^{x_1, \ldots, x_n}) = \lambda_{n, i + 1}((y_1, \ldots, y_n)^{x_1, \ldots, x_n}).$$

    In the expression of the image $l_n((y_1, \ldots, y_n)^{x_1, \ldots, x_n})$ as a sum the summands $\lambda_{n, i}((y_1, \ldots, y_n)^{x_1, \ldots, x_n})$ and $\lambda_{n, i + 1}((y_1, \ldots, y_n)^{x_1, \ldots, x_n})$ comes with opposite signs. Hence $$l_n((y_1, \ldots, y_n)^{x_1, \ldots, x_n})\in W_{n - 1}(\mathcal{X}, \mathcal{Y}).$$
\end{proof}

Let $C_n(\mathcal{X}, \mathcal{Y}) = V_n(\mathcal{X}, \mathcal{Y}) / W_n(\mathcal{X}, \mathcal{Y})$. It follows from the lemma \ref{Lemma:W} that the boundary homomorphisms $\partial_n\colon V_n(\mathcal{X}, \mathcal{Y})\to V_{n - 1}(\mathcal{X}, \mathcal{Y})$ induces a correctly defined homomorphism from $C_n(\mathcal{X}, \mathcal{Y})$ to $C_{n - 1}(\mathcal{X}, \mathcal{Y})$. Denote this homomorphism by the same letter $\partial_n$. Then the sequence of groups and homomorphisms $$\dots \to C_n(\mathcal{X}, \mathcal{Y})\xrightarrow{\partial_n} C_{n - 1}(\mathcal{X}, \mathcal{Y})\to\ldots$$ is a chain complex. Denote it $\mathcal{C}(\mathcal{X}, \mathcal{Y})$. Denote the corresponding cochain complex $\mathcal{C}^*(\mathcal{X}, \mathcal{Y})$ with cochain groups $C^n(\mathcal{X}, \mathcal{Y}) = \{\varphi\colon C_n(\mathcal{X}, \mathcal{Y})\to \mathbb{K}\}$ and coboundary homomorphisms $\partial^n\colon C^{n - 1}(\mathcal{X}, \mathcal{Y})\to C^n(\mathcal{X}, \mathcal{Y})$ defined by the formula $$\partial^n (\omega)((y_1, \ldots, y_n)^{x_1, \ldots, x_n}) = \omega(\partial_n(y_1, \ldots, y_n)^{x_1, \ldots, x_n})$$ for all $\omega\in C^n(\mathcal{X}, \mathcal{Y})$. Denote the cohomology groups of the cochain complex $\mathcal{C}^{*}(\mathcal{X}, \mathcal{Y})$ by $H^n(\mathcal{X}, \mathcal{Y})$.

\begin{remark}
    If $\omega\in C^2(\mathcal{X}, \mathcal{Y})$ is 3-cocycle, i.e. $\omega\in \ker\partial^3$, then the following identities holds:
    \begin{enumerate}
        \item $\omega((y, y)^{x, x}) = 0$ for all $x\in X$ and $y\in Y$;
        \item $\omega((y_1, y_2)^{x_1, x_2}) + \omega((y_1 *_{x_1}^{x_2} y_2, y_3)^{x_1 * x_2, x_3}) = \omega((y_1, y_3)^{x_1, x_3}) + \omega((y_1 *_{x_1}^{x_3} y_3, y_2 *_{x_2}^{x_3} y_3)^{x_1*x_3, x_2*x_3})$ for all $x_1, x_2, x_3\in X$ and $y_1, y_2, y_3\in Y$.
    \end{enumerate}
\end{remark}

\subsection{Cocycle invariant}\label{Subsection:CocycleInvariant} Let $(K, \varphi)$ be the oriented coloured link, $(D, \xi)$ the diagram of this link, and $\zeta\colon\mathcal{A}(D)\to Y$ the colouring of this diagram by the hierarchical quandle $\mathcal{Y}$. Fix a 2-cocycle $\omega\in \ker\partial^3$ of the cochain complex $\mathcal{C}^{*}(\mathcal{X}, \mathcal{Y})$. Then, for each crossing $v$ of the diagram $D$ with neighbourhood arcs $a_1, a_2, a_3\in \mathcal{A}(D)$ (as shown in the figure \ref{Figure:Crossings}), define the weight $$w_{\omega}(v) = \pm \omega((y_1, y_2)^{x_1, x_2}),$$ where $x_i = \xi(a_i)$, $y_i = \zeta(a_i)$, $i = 1, 2$, and we choose the sign $+$ if $v$ is positive, and $-$ if $v$ is negative (figure \ref{Figure:Weight}).

\begin{figure}[h]
    \begin{center}
        \begin{tikzpicture}[scale=0.75]
            \draw[knot_arrow, knot_diagram] (-1, -1) -- (1, 1) node[pos=0.15, right, xshift=2.0] {$y_2$};
            \draw (-1.6, -1.6) rectangle ++(0.6, 0.5) node[pos=.5] {$x_2$};

            \draw[knot_diagram] (1, -1) -- (0.1, -0.1) node[above right, pos=0.4, xshift=-3.0] {$y_1$};
            \draw (1.1, -1.6) rectangle ++(0.6, 0.5) node[pos=.5] {$x_1$};

            \draw[knot_arrow, knot_diagram] (-0.1, 0.1) -- (-1, 1);

            \draw[thin, ->] (0, -1.0) -- (0, -2.0) node[below] {$\omega((y_1, y_2)^{x_1, x_2})$};
        \end{tikzpicture}
        \hspace{2cm}
        \begin{tikzpicture}[scale=0.75]
            \draw[knot_arrow, knot_diagram] (1, -1) -- (-1, 1) node[pos=0.3, right] {$y_2$};
            \draw (1.1, -1.6) rectangle ++(0.6, 0.5) node[pos=.5] {$x_2$};

            \draw[knot_diagram] (-1, -1) -- (-0.1, -0.1);

            \draw[knot_arrow, knot_diagram] (0.1, 0.1) -- (1, 1) node[pos=0.6, left] {$y_1$};
            \draw (1.1, 1.1) rectangle ++(0.6, 0.5) node[pos=.5] {$x_1$};

            \draw[thin, ->] (0, -1.0) -- (0, -2.0) node[below] {$-\omega((y_1, y_2)^{x_1, x_2})$};
        \end{tikzpicture}
    \end{center}
    \caption{\label{Figure:Weight} The weight of the crossing $v$ of the coloured diagram}
\end{figure}
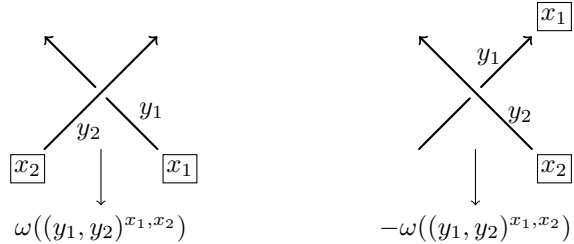

The weight of the coloured diagram $w_{\omega}(D, \xi, \zeta)$ is a sum $$w_{\omega}(D, \xi, \zeta) = \sum\limits_{v} w_{\omega}(v),$$ where the sum is taken over all the crossings $v$ of the diagram $D$.

\begin{theorem}
    \label{Theorem:CocycleInvariant}
    Let $(D_1, \xi_1)$ and $(D_2, \xi_2)$ be two diagrams of the oriented coloured link $(K, \varphi)$, and let $\zeta_i\colon\mathcal{A}(D_i)\to Y$, $i = 1, 2$, be corresponding colourings of these diagrams by the hierarchical quandle $\mathcal{Y}$. Then for any 2-cocycle $\omega\in\ker\partial^3$ $$w_{\omega}(D_1, \xi_1, \zeta_1) = w_{\omega}(D_2, \xi_2, \zeta_2).$$
\end{theorem}
\begin{proof}
    It's sufficient to prove the theorem only in the case when the diagram $(D_2, \xi_2)$ is obtained from the diagram $(D_1, \xi_1)$ by a Reidemeister move (figures \ref{Figure:R1}, \ref{Figure:R2} and \ref{Figure:R3}).

    \emph{First Reidemeister move}. Let $v$ be the crossing of the diagram $D_2$, obtained after this move. Then $w_{\omega}(v) = 0$, because $\omega((y, y)^{x, x}) = 0$. So $w_{\omega}(D_1, \xi_1, \zeta_1) = w_{\omega}(D_2, \xi_2, \zeta_2)$.

    \emph{Second Reidemeister move}. Let $v_1, v_2$ be two crossings of the diagram $D_2$ obtained after this move. The first is positive, the other negative. Then $w_{\omega}(v_1) = -w_{\omega}(v_2)$. So again $w_{\omega}(D_1, \xi_1, \zeta_1) = w_{\omega}(D_2, \xi_2, \zeta_2)$.

    \emph{Third Reidemeister move}. Let $v_1, v_2, v_3$ be crossings of $D_1$ before this move, and let $u_1, u2_, u_3$ be crossings of $D_2$ after this move. Then 
    \begin{multline*}
        w_{\omega}(v_1) + w_{\omega}(v_2) + w_{\omega}(v_3) = \\ = \omega((y_2, y_3)^{x_2, x_3}) + \omega((y_1, y_3)^{x_1, x_3}) + \omega((y_1 *_{x_1}^{x_3} y_3, y_2 *_{x_2}^{x_3}y_3)^{x_1*x_3, x_2*x_3}),
    \end{multline*}
    \begin{multline*}
        w_{\omega}(u_1) + w_{\omega}(u_2) + w_{\omega}(u_3) = \\ = \omega((y_1 *_{x_1}^{x_2} y_2, y_3)^{x_1*x_2, x_3}) + \omega((y_1, y_2)^{x_1, x_2}) + \omega((y_2, y_3)^{x_2, x_3}).
    \end{multline*}

    Since $\omega$ is the 2-cycle, these sums are equal. So again $w_{\omega}(D_1, \xi_1, \zeta_1) = w_{\omega}(D_2, \xi_2, \zeta_2)$.
\end{proof}

It follows from the theorem \ref{Theorem:CocycleInvariant} that for any coloured link $(K, \varphi)$ and any diagram $(D, \xi)$ of this link, the multi-set of weights $w_{\omega}(D, \xi, \zeta)$ computed for all different colourings $\zeta$ is an invariant of the link. Denote the value of this invariant $\mathcal{W}_{\omega}(K, \varphi)$.

\begin{theorem}
    \label{Theorem:HomologicCocycles}
    Let $(K, \varphi)$ be an oriented coloured link. Then for any two homological 2-cocycles $\omega_1, \omega_2\in\ker \partial^3$: $$\mathcal{W}_{\omega_1}(K, \varphi) = \mathcal{W}_{\omega_2}(K, \varphi).$$
\end{theorem}
\begin{proof}
    It's sufficient to prove that if $\omega\in \text{Im\;} \partial^2$, then $\mathcal{W}_{\omega}(K, \varphi) = 0$. Let $\omega = \partial^2(\delta)$ for some $\delta\in C^1(\mathcal{X}, \mathcal{Y})$. Then $$\omega((y_1, y_2)^{x_1, x_2}) = \delta((y_1 *_{x_1}^{x_2} y_2)^{x_1 * x_2}) - \delta((y_1)^{x_1}).$$

    Note that the weight $w_{\omega}(v)$ of any crossing $v$ is equal to the difference $\delta((\alpha)^{a}) - \delta((\beta)^{b})$, where $a$ is the quandle colour and $\alpha$ is the hierarchical quandle colour of the outgoing lower arc, and $b$ is the quandle colour and $\beta$ is the hierarchical quandle colour of the incoming lower arc (figure \ref{Figure:LowerArcs}).

    \begin{figure}[h]
        \begin{center}
            \begin{tikzpicture}[scale=0.75]
                \draw[knot_arrow, knot_diagram] (-1, -1) -- (1, 1);
                \draw[knot_diagram] (1, -1) -- (0.1, -0.1) node[pos=0.5, right] {$\beta$};
                \draw[knot_arrow, knot_diagram] (-0.1, 0.1) -- (-1, 1) node[pos=0.5, left] {$\alpha$};

                \draw (0.8, -1.5) rectangle ++(0.4, 0.4) node[pos=0.5] {$b$};
                \draw (-1.2, 1.1) rectangle ++(0.4, 0.4) node[pos=0.5] {$a$};
            \end{tikzpicture}
            \hspace{2cm}
            \begin{tikzpicture}[scale=0.75]
                \draw[knot_arrow, knot_diagram] (1, -1) -- (-1, 1);
                \draw[knot_diagram] (-1, -1) -- (-0.1, -0.1) node[pos=0.5, left] {$\beta$};
                \draw[knot_arrow, knot_diagram] (0.1, 0.1) -- (1, 1) node[pos=0.5, right] {$\alpha$};

                \draw (-1.2, -1.5) rectangle ++(0.4, 0.4) node[pos=0.5] {$b$};
                \draw (0.8, 1.1) rectangle ++(0.4, 0.4) node[pos=0.5] {$a$};
            \end{tikzpicture}
        \end{center}
        \caption{\label{Figure:LowerArcs}Colours of the lower arcs in the neighbourhood of the crossing}
    \end{figure}
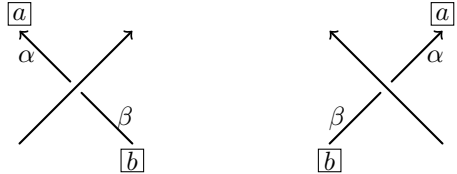

    Therefore the sum over all crossings of the weights $w_{\omega}(v)$ is zero.
\end{proof}

It follows from the theorem \ref{Theorem:HomologicCocycles} that the value of the invariant $\mathcal{W}_{\omega}$ depends only on the cohomological class $[\omega]\in H^2(\mathcal{X, \mathcal{Y}})$.

\section{Quandles from hierarchical quandles}\label{Section:QuandlesForHierarchicalQuandles}

Let $\mathcal{X} = (X, *)$ be a quandle, and let $\mathcal{Y} = (Y, \{*_{x_1}^{x_2}\}_{x_1, x_2\in X})$ be a hierarchical quandle over $\mathcal{X}$. Define the quandle $\mathcal{Q}_{\mathcal{X}, \mathcal{Y}} = (X\times Y, \circ)$, where
\begin{center}
    $(x_1, y_1) \circ (x_2, y_2) = (x_1 * x_2, y_1 *_{x_1}^{x_2} y_2)$ for any $x_!, x_2\in X$ and $y_!, y_2\in Y$.
\end{center}

\begin{theorem}
    \label{Theorem:XY}
    For any quandle $\mathcal{X}$ and hierarchical quandle $\mathcal{Y}$ over $\mathcal{X}$ the quandle $\mathcal{Q}_{\mathcal{X}, \mathcal{Y}}$ is indeed the quandle.
\end{theorem}
\begin{proof} We should check three axioms from the definition \ref{Definition:Quandle}.

    \emph{The first axiom}. For any $x\in X$ and $y\in Y$:
    $$(x, y)\circ (x, y) = (x * x, y *_{x}^{x} y) = (x, y).$$

    \emph{The second axiom}. Let $(a, b), (\alpha, \beta)\in X\times Y$. We should check that there is a unique $(x, y)\in X\times Y$ such that $$(x, y)\circ (a, b) = (\alpha, \beta).$$

    Indeed, $$(x, y)\circ (a, b) = (x * a, y *_{x}^{a} b).$$ Since $\mathcal{X}$ is a quandle, the equation $x * a = \alpha$ has a unique solution $x = \alpha / a$. Since $\mathcal{Y}$ is a hierarchical quandle, the equation $y *_{x}^{a} b = \beta$ also has a unique solution $y = \beta /_{x}^{a} b$.

    \emph{The third axiom}. For any $(x_1, y_1), (x_2, y_2), (x_3, y_3)\in X\times Y$:
    \begin{multline*}
        ((x_1, y_1)\circ (x_2, y_2))\circ (x_3, y_3) = (x_1 * x_2, y_1 *_{x_1}^{x_2} y_2) \circ (x_3, y_3) = \\ = ((x_1 * x_2) * x_3, (y_1 *_{x_1}^{x_2} y_2) *_{x_1 * x_2}^{x_3} y_3) = \\ = ((x_1 * x_3) * (x_2 * x_3), (y_1 *_{x_1}^{x_3} y_3) *_{x_1 * x_3}^{x_2 * x_3} (y_2 *_{x_2}^{x_3} y_3)) = \\ = (x_1 * x_3, y_1 *_{x_1}^{x_3} y_3) \circ (x_2 * x_3, y_2 *_{x_2}^{x_3} y_3) = ((x_1, y_1)\circ (x_3, y_3))\circ ((x_2, y_2)\circ (x_3, y_3)).
    \end{multline*}
\end{proof}

\begin{remark}
    It's clear that if $\mathcal{X}$ is a finite quandle and $\mathcal{Y}$ is a finite hierarchical quandle over $\mathcal{X}$, then for any oriented link $K$ $$|C_{\mathcal{Q}_{\mathcal{X}, \mathcal{Y}}}(K)| = \sum\limits_{\varphi}|C_{\mathcal{Y}}(K, \varphi)|,$$ where the sum is taken over all colourings $\varphi$ of the link $K$ by the quandle $\mathcal{X}$.

    Furthermore, the cohomology group $H^2(\mathcal{X}, \mathcal{Y})$ is isomorphic to the classical second cohomology group $H^2(\mathcal{Q}_{\mathcal{X}, \mathcal{Y}})$ of the quandle $\mathcal{Q}_{\mathcal{X}, \mathcal{Y}}$.

    In general multi-sets of invariants $\{C_{\mathcal{Y}}(K, \varphi)| \varphi\in C_{\mathcal{X}}(K)\}$ and $\{\mathcal{W}_{\omega}(K, \varphi)| \varphi\in C_{\mathcal{X}}(K)\}$ are more powerful than quandle colourings invariant $C_{\mathcal{Q}_{\mathcal{X}, \mathcal{Y}}(K)}$ and cocycle invariant for the quandle $\mathcal{Q}_{\mathcal{X}, \mathcal{Y}}$.
\end{remark}

\section{Further development}\label{Section:FurtherDevelopment}

In this section we describe some tasks that look interesting in the context of hierarchical quandles.

\subsection{Examples} Construct an infinite series of hierarchical quandles over quandles of some special types. It's interesting to find structures that allows to construct new hierarchical quandles. Something like \cite{IIJO} or \cite{AG}.

\subsection{Essential invariants} The constructions of the invariants $C_{\mathcal{Y}}$ and $\mathcal{W}_{\omega}$ are not interesting, because these invariants are mostly the same as the classical quandle colourings and quandle cocycle invariants from \cite{CJKLS}. Is it possible to construct the invariant that is essentially new?

\subsection{Positive cocycle invariants} In the definition of the chain (and corresponding cochain) complex $\mathcal{C}(\mathcal{X}, \mathcal{Y})$ we used the boundary homomorphisms $\partial_n = l_n - r_n$. It's possible to construct another version of the chain complex by using the boundary homomorphisms $\delta_n = l_n + r_n$. It follows from the lemma \ref{Lemma:LR} that $\delta_{n - 1}\circ \delta_n = 0$. Two-cocycles of the corresponding cochain complexes lead to the positive cocycle invariants, which are similar to the invariants from \cite{CG}.

\subsection{Hierarchical quasoids} The notion of the \emph{quasoid} was introduced in \cite{K1}. It's equivalent to the more earlier notion of \emph{ternary algebra} from \cite{N1} and \emph{Niebrzydowski's tribracket} from \cite{NOO, NP}. In \cite{K2} and \cite{N2} the theory of cocycle invariants for quasoids is developed. It's possible to define hierarchical quasoids over other quasoids and to construct cocycle invariants for these hierarchical quasoids. It's also possible to define hierarchical quasoids over quandles and hierarchical quandles over quasoids. It's interesting to study the corresponding cocycle invariants.

Another task is to study the connection between hierarchical quandle/quasoid over quandle/quasoid cocycle invariants and mixed invariants from \cite{K3}.

\subsection{Quantum invariants for coloured links} Let $\mathcal{X} = (X, *)$ be a finite quandle. Fix the family $\{V_x\}_{x\in X}$ of finite dimensional modules, indexed by elements of $X$. Consider the family of operators $R_{x_1, x_2}\colon V_{x_1}\otimes V_{x_2}\to V_{x_2}\otimes V_{x_1 * x_2}$, $x_1, x_2\in X$. This family $\{R_{x_1, x_2}\}_{x_1, x_2\in X}$ is called \emph{coloured Yang -- Baxter operator} if for all $x_1, x_2, x_3\in X$ 
\begin{multline*}
    (R_{x_2, x_3}\otimes id_{(x_1*x_2)*x_3})\circ (id_{x_2}\otimes R_{x_1*x_2, x_3})\circ (R_{x_1, x_2}\otimes id_{x_3}) = \\ = (id_{x_3}\otimes R_{x_1*x_3, x_2*x_3})\circ (R_{x_1, x_3}\otimes id_{x_2*x_3})\circ (id_{x_1}\otimes R_{x_2, x_3}).
\end{multline*}

Each coloured Yang -- Baxter operator assigns to each braid coloured by the quandle $\mathcal{X}$ the homomorphism between the tensor product of suitable modules. It's interesting to find a trace operator for these homomorphisms such that the corresponding value doesn't change under Markov moves of braids. This will allow to construct quantum invariants of coloured links similar to the invariants from \cite{T}.

\end{document}